\documentclass{amsart}
\usepackage[utf8]{inputenc}
\usepackage{amsmath,amssymb,amsfonts,mathtools}
\usepackage{hyperref}
\usepackage{enumitem}
\usepackage{comment}
\usepackage{xcolor}
\usepackage{ulem}
\usepackage{graphicx}
\usepackage{placeins}

\newtheorem{theorem}{Theorem}

\newtheorem{lemma}{Lemma}

\newtheorem{problem}{Problem}
\theoremstyle{definition}

\theoremstyle{remark}
\newtheorem{remark}{Remark}

\newcommand{\R}{\mathbb{R}}
\newcommand{\C}{\mathbb{C}}
\newcommand{\spec}{\operatorname{Spec}}

\title[Periodic Orbits in Polynomial Hurwitz Vector Fields]{A family of Hurwitz polynomial vector fields in $\R^3$\\ having non-constant periodic orbits}

\author[\'A. Casta\~neda]{\'Alvaro Casta\~neda}
\author[S. Rebollo-Perdomo]{Salom\'on Rebollo-Perdomo}

\address{Departamento de Matem\'aticas, Universidad de Chile. Casilla 653, Santiago, Chile}
\email{castaneda@uchile.cl}

\address{Departamento de Matem\'atica, Universidad del B\'{\i}o-B\'{\i}o, Casilla 5--C, Concepci\'on, Chile}
\email{srebollo@ubiobio.cl}


\thanks{The authors are funded by FONDECYT Regular Grant 1240361. The second author also is supported by Universidad de B\'{\i}o-B\'{\i}o Grant RE2482302}

\date{}
\begin{document}
\begin{abstract}

We construct a two-parameter family of polynomial vector fields in $\R^{3}$ and prove that, on a nonempty open set of parameter values, every member of the family is Hurwitz and possesses two non-constant periodic orbits, one of saddle type and one attracting.  
This answers Problem~19 in Gasull's list of open problems on low-dimensional dynamical systems concerning the existence of non-constant periodic orbits in Hurwitz smooth vector fields on $\R^3$. This also 
provides counterexamples to the expected non-existence of non-constant periodic orbits for Hurwitz polynomial vector fields; a question that arose from the first counterexamples to the Markus--Yamabe conjecture in the late 1990s and that was stated by Meisters in his featured MathSciNet review of the Bernat--Llibre counterexample.

\end{abstract}

\maketitle

\section{Introduction}

In 1960, Markus and Yamabe \cite{MY} asked whether a pointwise spectral
condition on the Jacobian matrix of a vector field is sufficient to guarantee
global asymptotic stability. Their question, now known as the Markus--Yamabe
Conjecture, can be stated as follows.

\medskip
\begin{center}
\begin{minipage}{0.9\textwidth}
\textbf{The Markus--Yamabe Conjecture.}
Consider the differential system
\begin{equation*}
    \dot{\mathbf{x}}=\mathbf{F}(\mathbf{x}), \qquad
    \mathbf{x}\in\R^{n},
\end{equation*}
where $\mathbf{F}\colon\R^{n}\to\R^{n}$ is of class $C^{1}$ and
$\mathbf{F}(\mathbf{0})=\mathbf{0}$. Assume that, for every
$\mathbf{x}\in\R^{n}$,
\begin{equation*}
    \spec D\mathbf{F}(\mathbf{x})
    \subset\{\lambda\in\C:\operatorname{Re}\lambda<0\}.
\end{equation*}
Then the singular point $\mathbf{x}=\mathbf{0}$ is globally asymptotically stable.
\end{minipage}
\end{center}
\medskip

A vector field $\mathbf{F}(\mathbf{x})$ satisfying the spectral condition is called \textit{Hurwitz}. The
conjecture is true for $n\leq 2$. The planar case was settled independently by
Fe\ss ler \cite{Fessler}, Guti\'errez \cite{Gutierrez}, and Glutsyuk
\cite{Glutsyuk}, after the earlier proof of Meisters and Olech
\cite{MeistersOlech} for planar polynomial vector fields. For $n\geq 3$ the
conjecture is false. In particular, Cima \textit{et al.} \cite{CEGHM} constructed a polynomial Hurwitz vector field in
$\R^{3}$ whose associated differential system has a solution escaping to infinity.

Although the Markus--Yamabe Conjecture was completely settled in the (classical) autonomous setting, it
continues to be revisited from several points of view.
Continuous and discontinuous piecewise linear differential systems have been
studied by Llibre and Zhang \cite{Llibre1}, Llibre and Menezes \cite{Llibre2},
and Zhang and Yang \cite{ZhangYang}. An infinite-dimensional point of view was
adopted by Rodrigues, Teixeira, and Gameiro \cite{Rodrigues}. In a
nonautonomous framework, Cheban \cite{Cheban} worked with cocycles, while
Casta\~neda and Robledo \cite{CR} introduced a spectral dichotomy version of
the problem. Several authors have also studied polynomial and nilpotent classes
closely related to the Markus--Yamabe Conjecture, together with their global
dynamics; see, for instance, \cite{CastanedaRebollo, CastanedaMachado, CGM}.
Two remarkable settings in which the conjecture does hold are provided by
triangular and gradient vector fields: Markus and Yamabe proved the conjecture
for triangular systems \cite[Theorem~4]{MY}, while Hartman
\cite[p.~539, Corollary~11.2]{Hartman}, together with the subsequent result of
Ma\~nosas and Peralta-Salas \cite{Manosas}, treated the gradient case.

In the present article, we do not seek a
new class of vector fields for which the Markus--Yamabe Conjecture holds, nor do we propose a new formulation of global asymptotic stability. Rather, the emphasis here is different.
Our aim is to understand
one specific mechanism that may destroy global attraction under the Hurwitz
hypothesis: the existence of periodic orbits. More precisely, we address the
problem of determining whether a Hurwitz polynomial vector field in dimension
three can have non-constant periodic orbits.

\subsection{From the first counterexamples to the periodic orbit problem}
\label{subsec:counterexamples}

Two counterexamples of very different nature to the Markus--Yamabe Conjeture are essential for framing the question
considered in this work. First, Bernat and Llibre \cite{BernatLlibre}
constructed a four-dimensional analytic Hurwitz differential system having a
non-constant periodic solution. Their example showed that the failure of global
asymptotic stability can be produced by periodic 
 orbits.
Second, Cima
\textit{et al.} \cite{CEGHM} found the three-dimensional polynomial vector field
\begin{equation*}
\mathbf{F}(x,y,z)=\left(-x+z(x+yz)^2,\,-y-(x+yz)^2,\,-z\right),
\end{equation*}
for which $D\mathbf{F}(x,y,z)$ has $-1$ as its only eigenvalue at every point.
Nevertheless, the corresponding system has the unbounded solution
\begin{equation*}
\label{eq:cimaorbit}
    \left(18 e^{t},-12 e^{2t},e^{-t}\right),
   \quad  t\in \R.
\end{equation*}
Thus polynomiality does not recover the global stability conclusion in
dimension three. However, the two mechanisms are completely
different: 
the former has a non-constant periodic
orbit, whereas the latter has
an orbit escaping to infinity.

We emphasize that this distinction was pointed out immediately after these
counterexamples appeared. In his featured MathSciNet review of the paper of
Bernat and Llibre, Meisters \cite{MeistersReview} discussed in detail both the
four-dimensional analytic construction and the 
discovered
three-dimensional polynomial example by Cima \textit{et al.} After stressing that the Bernat--Llibre
system remained important because of its periodic orbit, he stated that
\begin{quote}
``such a non-constant periodic orbit is surely impossible for polynomial
Hurwitz systems. (This remains to be proved.)''
\end{quote}
In this way, the review raised a problem that is more precise than the
original Markus--Yamabe Conjecture. Once the polynomial Hurwitz counterexamples
were known, the question was no longer whether global attraction could fail,
but whether such a failure could be produced by periodic dynamics inside the
polynomial category.

The difference is particularly relevant in dimension three. In dimension two 
 the periodic phenomenon is excluded. Bernat and
Llibre showed that periodic dynamics is possible for analytic Hurwitz systems
in dimension four.
What
remained open was the simultaneous occurrence of the three properties
\[
\text{dimension three},\qquad \text{ Hurwitz polynomial  vector field},\qquad
\text{non-constant periodic orbit}.
\]

More than two decades later, Gasull formalized the three-dimensional
periodic orbit question in his list of open problems in low-dimensional
dynamical systems \cite{Gasull} as follows.

\begin{problem}[{\cite[Problem 19]{Gasull}}]
\label{prob:gasull}
Are there smooth vector fields in $\R^{3}$ under the hypotheses of the
Markus--Yamabe conjecture and having periodic orbits?
\end{problem}

We point out that Problem~\ref{prob:gasull} only asks for a smooth vector
field. In view of the discussion above, the polynomial case is a stronger
question, and it is precisely the one suggested by the featured review
\cite{MeistersReview}. The aim of this work is to answer this stronger
problem.

\subsection{Main result}

The main result of this paper gives an affirmative answer to
Problem~\ref{prob:gasull} and, at the same time, shows that the expectation
stated in \cite{MeistersReview}  does not hold. 
The construction is stronger than a
single isolated example: it provides a two-parameter family of polynomial vector fields and an
open set of parameters for which the Hurwitz property and the periodic
dynamics coexist. More precisely, 
we consider the following 2-parameter family of polynomial vector fields:
\begin{equation}
\label{campo-principal}
\mathbf{F}_{\mu,e}(x,y,z):=\big(F_1(x,y,z), F_2(x,y,z), F_3(x,y,z)\big),
\end{equation}
where $\mu,e\in (0, \infty)$,
\begin{align*}
F_1(x,y,z)&=\mu x\left(  \left( {x}^{2}+{y}^{2}-5 \right)  \left( {\frac {20}
{21}}-{x}^{2}-{y}^{2} \right) -{\frac {400\,{z}^{2}}{21}} \right) -
 \left( 1+e\bigg(z+\frac{ {x}^{2}+{y}^{2}}{10} \bigg)  \right) y,
\\
F_2(x,y,z)&=\mu y\, \left(  \left( {x}^{2}+{y}^{2}-5 \right)  \left( {\frac {20}
{21}}-{x}^{2}-{y}^{2} \right) -{\frac {400\,{z}^{2}}{21}} \right) +
 \left( 1+e\bigg(z+\frac{ {x}^{2}+{y}^{2}}{10} \bigg)  \right) x,
\\
F_3(x,y,z)&=-z-\frac{ {x}^{2}+{y}^{2}}{10} .
\end{align*}

We prove the following result.

\begin{theorem}
\label{main-theorem}
There exists an small open region $U$ in the square $(0,1/500)\times (0,1/6)$  such that for each pair $(\mu,e)\in U$ the corresponding polynomial vector field $\mathbf{F}_{\mu,e}(x,y,z)$ given in \eqref{campo-principal} is Hurwitz on $\R^3$, it has a unique singular point at the origin, which is locally asymptotically stable, and exactly two hyperbolic $2\pi$-periodic orbits
\[
\gamma_1:=\left\{(x,y,z)\in \R^3 \mid x^2+y^2=1,\ z=-\frac{1}{10}\right\},
\quad
\gamma_2:=\left\{(x,y,z) \in \R^3\mid x^2+y^2=4,\ z=-\frac{2}{5}\right\}.
\]
Moreover,  $\gamma_1$ is 
 of saddle type and $\gamma_2$ is locally asymptotically stable. 
\end{theorem}

The novelty of Theorem~\ref{main-theorem} lies in the dynamical mechanism that destroys global attraction. The example
of Cima \textit{et al.} destroys global attraction through an unbounded solution,
while the vector fields introduced in this work do it through periodic orbits.
Therefore, the phenomenon exhibited by Bernat and Llibre in the analytic
four-dimensional setting already occurs for polynomial vector fields in
dimension three. 
In fact, we show that there are Hurwitz polynomial vector fields in $\R^n$, for $n\geq 3$, having periodic orbits; see Remark~\ref{rem-nD-vec-field}. 
We point out that the periodic orbit $\gamma_2$ is actually an attracting limit cycle (isolated periodic orbit) of the vector field, showing the coexistence of two different types of attractors in Hurwitz polynomial vector fields.

\medskip

The core of the proof is to establish the Hurwitz property for $\mathbf{F}_{\mu,e}$ on the whole space $\R^{3}.$ If
\[
P(\lambda)=\lambda^3+a_1(\mathbf{x})\lambda^2+a_2(\mathbf{x})\lambda+a_3(\mathbf{x})
\]
denotes the characteristic polynomial of $D\mathbf{F}_{\mu,e}(\mathbf{x})$, then the Hurwitz property for $\mathbf{F}_{\mu,e}$ is
equivalent to the Routh--Hurwitz inequalities
\[
a_1(\mathbf{x})>0,\qquad a_3(\mathbf{x})>0,\qquad a_1(\mathbf{x})\,a_2(\mathbf{x})-a_3(\mathbf{x})>0,
\]
which must be verified for every $\mathbf{x}=(x,y,z)\in\R^{3}$. By construction, the familiy has a rotational symmetry around $z-$axis, which reduces this verification to the plane $\{y=0\}$. There, the first
two inequalities are proved by explicit estimates valid on a whole rectangle
of parameters, whereas the third one is first established 
 for 
$\mathbf{F}_{\mu,e}(\mathbf{x})$ at $(\mu_0,e_0) = (1/1000, 1/10)$, by means of an exact algebraic criterion for the number of
real roots of a polynomial (the discriminant sequences of Yang \cite{Yang}),
and  then extended to an open set of parameters around $(\mu_0,e_0)$
by a compactness argument, after proving that the inequality holds
outside a compact set uniformly with respect to the parameters.  The
existence, exact number, and transverse stability of the periodic orbits are
established independently through a reduced vector field in cylindrical
coordinates.

This separation is useful for understanding the construction. On the one hand,
the Cartesian Jacobian matrix is Hurwitz at every point of the phase space. On
the other hand, the global geometry of the vector field produces periodic
dynamics. Hence the pointwise Hurwitz condition and hyperbolic periodic orbits
can coexist in dimension three, even under the additional restriction of
polynomiality.

\subsection{Some consequences and perspectives}

Theorem~\ref{main-theorem} closes the gap between the two classical mechanisms
described above. The Bernat--Llibre example provides periodic dynamics but is
analytic and four-dimensional; the Cima--van den
Essen--Gasull--Hubbers--Ma\~nosas example is polynomial and three-dimensional
but  an unbounded orbit. The family constructed here exhibits
simultaneously the three properties listed in
Section~\ref{subsec:counterexamples}: dimension three, Hurwitz polynomial  vector field, and
non-constant periodic orbit.

We point out that this result also changes the role of the periodic orbit
question inside the Markus--Yamabe problem. After the classical conjecture was
disproved, it was natural to ask which global behaviors were still excluded by
the pointwise Hurwitz condition. Problem~19 of Gasull identifies periodic
orbits as one of the first such behaviors to test. Theorem~\ref{main-theorem}
shows that they are not excluded, even in the polynomial setting.

The construction naturally leads to further questions. The open parameter
region $U,$ mentioned in Theorem~\ref{main-theorem}, shows that the phenomenon is not confined to a single
member of the family. Moreover, the two periodic orbits are hyperbolic, so they
have the usual local $C^1$ persistence as dynamical objects. There are
distinct statements: hyperbolicity gives persistence of the periodic orbits,
whereas the global Routh--Hurwitz argument is what guarantees the Hurwitz
property throughout an open region of the $(\mu,e)$-plane. It would be
interesting to determine how large the class of Hurwitz polynomial vector
fields exhibiting this behavior is, and which additional structural
assumptions are sufficient to exclude periodic dynamics. These broader
questions will not be pursued here.

\subsection{Structure of the paper} The paper is organized as follows. In Section~\ref{sec:field} we give the proof of Theorem~\ref{main-theorem},
which is based on three lemmas: Lemma~\ref{propiedades-dinamicas}, on
the singular point and the periodic orbits;
Lemma~\ref{invarianza-de-matriz-jacobiana}, on the invariance of the
characteristic polynomial of the Jacobian matrix under rotations around the
$z$-axis; and Lemma~\ref{Hurwitz-en-plano-xz}, on the Hurwitz property of $\mathbf{F}_{\mu_0,e_0}(\mathbf{x})$. Section~\ref{proofs} contains the proofs of the Lemmas.
Section~\ref{example} presents a detailed analysis of the dynamics of the vector field $\mathbf{F}_{\mu_0,e_0}(\mathbf{x})$, together with a numerical phase portrait that illustrates the theorem.

\section{Proof of Main Result}\label{sec:field}

The proof of Theorem \ref{main-theorem} will divide in the following three lemmas, which we prove in the following section. The first lemma concerns the dynamical properties of the vector field $\mathbf{F}_{\mu,e}(x,y,z)$.

\begin{lemma}
\label{propiedades-dinamicas}
Each polynomial vector field $\mathbf{F}_{\mu,e}(x,y,z)$ given in \eqref{campo-principal}, with  $0<\mu\leq 588/6575$, has a unique singular point at the origin, which is  locally asymptotically stable, and exactly two hyperbolic $2\pi$-periodic orbits
\[
\gamma_1:=\left\{(x,y,z)\in \R^3 \mid x^2+y^2=1,\ z=-\frac{1}{10}\right\},
\quad
\gamma_2:=\left\{(x,y,z) \in \R^3\mid x^2+y^2=4,\ z=-\frac{2}{5}\right\}.
\]
Moreover,  $\gamma_1$ is of saddle type and $\gamma_2$ is locally asymptotically stable.
\end{lemma}

The second and third lemmas are regarding the Hurwitz property of  $\mathbf{F}_{\mu,e}(x,y,z)$. 
\begin{lemma}
\label{invarianza-de-matriz-jacobiana}
The characteristic polynomial of the Jacobian matrix $D\mathbf{F}_{\mu,e}(x,y,z)$ of the polynomial vector field $\mathbf{F}_{\mu,e}(x,y,z)$ given in \eqref{campo-principal} is invariant under  rotations around the $z$-axis.
\end{lemma}

\begin{lemma}
\label{Hurwitz-en-plano-xz}
Let $\mu=\mu_0=1/1000$ and $e=e_0=1/10$.  The Jacobian matrix $D\mathbf{F}_{\mu_0,e_0}(x,y,z)$ is Hurwitz for every point $(x,0,z)\in\R^3$.
\end{lemma}

The second lemma implies that is enough to prove the Hurwitz property  for points of the form $(x,0,z)$. The third one ensures that a concrete polynomial vector field $\mathbf{F}_{\mu_0,e_0}(x,y,z)$ has Hurwitz property  for points of the form $(x,0,z)$.
The combination of these two last lemmas gives the Hurwitz property of $\mathbf{F}_{\mu_0,e_0}(x,y,z)$ on the whole $\R^3$.
We will prove Theorem~\ref{main-theorem} by assuming Lemmas \ref{propiedades-dinamicas}, \ref{invarianza-de-matriz-jacobiana}, and \ref{Hurwitz-en-plano-xz}.

\begin{proof}[Proof of Theorem~\ref{main-theorem}]
The uniqueness of the origin as singular point and the existence of the two periodic orbits $\gamma_1$ and $\gamma_2$, as well as their dynamical properties are proved by Lemma~\ref{propiedades-dinamicas}. Hence, it remains to prove the Hurwitz property of the vector field.

Lemma~\ref{invarianza-de-matriz-jacobiana} implies that it is enough to prove the Hurwitz property  for points of the form $(x,0,z)$.
The characteristic polynomial of $D\mathbf{F}_{\mu,e}(x,0,z)$ has the form
$$
P_{\mu,e}(\lambda)=\lambda^3+c_2\lambda^2
+c_1\lambda+c_0,
$$
where
\begin{align*}
c_2&=1+ \mu \left( 6\, \left( x^2-{\frac {125}{63}} \right) ^{2}+{\frac {800\,{z}
^{2}}{21}}-{\frac {18650}{1323}} \right), \\
c_1&=\frac{1}{100}\big( 2ex^2+g(x,z)\big)g(x,z)+\frac{\mu }{21}f_1(x,z)+\frac{5\mu^2}{441}f_2(x,z)f_3(x,z),\\
c_0&=\frac{1}{100}\big(g(x,z)\big)^2+\frac{5\mu^2 }{441}f_3(x,z)f_4(x,z),
\end{align*}
with 
\begin{align*}
g(x,z)&=ex^2+10ez+10,\\
f_1(x,z)&=126x^4-160x^2z-500x^2+800z^2+200,\\
f_2(x,z)&=21x^4-75x^2+80z^2+20,\\
f_3(x,z)&=21x^4-125x^2+400z^2+100,\\
f_4(x,z)&=21x^4-32x^2z-75x^2+80z^2+20.
\end{align*}
We prove the Hurwitz property of $D\mathbf{F}_{\mu,e}(x,0,z)$  by proving that $P_{\mu,e}(\lambda)$ is stable by using the Routh--Hurwitz criterion for cubic polynomials, that is, we will prove that $c_2>0, c_1>0, c_0>0,$ and $c_2c_1-c_0>0$. 
In fact, we need only to prove the inequalities 
$$
c_2>0,\ c_0>0, \quad \textrm{and} \quad d:=c_2c_1-c_0>0,
$$ 
because the last one immediately yields $c_1=(d+c_0)/c_2>0$.

\smallskip
A straightforward computation gives
$$
d=c_2c_1-c_0=d_0+\mu\Big(d_1+d_2\mu+d_3\mu^2\Big),
$$
where
\begin{align*}
d_0&=\frac{e}{50}x^2\big(g(x,z)\big)\\
d_1&=\frac{1}{21}\Bigg(400+400ez+
(200e^2+1600)z^2+1600ez^3+800e^2z^4
\\ &\qquad \qquad 
+(80e-1000)x^2
-20e(25e-16)z^2x^2+
(80e^2-1000e-160)zx^2
\\ &\qquad \qquad \quad 
+(6e^2-200e+252)x^4 +320e^2z^3x^2-4e(50e-63)zx^4\\
& \qquad \qquad \qquad 
+150e^2z^2x^4+\frac{252}{5}e^2zx^6-\frac{3}{5}e(25e-84)x^6+\frac{189}{50}e^2x^8\Bigg)\\
\end{align*}
\begin{align*}
d_2&=\frac{1}{441}\Bigg(
40000+320000z^2+640000z^4-64000z^3x^2-800000z^2x^2\\ 
&\qquad \qquad \quad
-16000zx^2-200000x^2+201600z^2x^4
+60000zx^4
\\ 
&\qquad \qquad \qquad 
+300400x^4-16800zx^6-
126000x^6+15876x^8\Bigg)\\
d_3&=\frac{10}{9261}\bigg(\Big(63x^4-250x^2+
400z^2+100\Big)f_2(x,z)f_3(x,z)\bigg)
\end{align*}

\medskip
We divide the proof into three steps. In Step 1 we prove that $c_2>0$ for suitable positive values of $\mu$. In Step 2 we prove that $c_0>0$ for suitable positive values of $\mu$ and $e$. Finally, according to the constraints we obtain 
for $\mu$ and $e$ in the first two steps, together with Lemma~\ref{Hurwitz-en-plano-xz}, we prove that $d>0$ for values in an open region $U\subset (0,1/500)\times(0,1/6)$ for which
all the roots of $P_{\mu,e}(\lambda)$ have negative real part.  

\medskip

{\bf Step 1.} The coefficient of $\mu$ in $c_2$ has two critical minima at $(\pm \sqrt{125/63},0)$ of the $xz$-plane, whose global minimum value is $-18650/1323$. Therefore, for $0<\mu<\mu_2^*:=1323/18650$, we get $c_2>0$ for any $(x,z)$.

\medskip
{\bf Step 2.}  The functions $f_3(x,z)$ and $f_4(x,z)$ can be written as
$$
f_3(x,z)=21\, \left( {x}^{2}-{\frac {125}{42}} \right) ^{2}+400\,{z}^{2}-{
\frac {7225}{84}}
$$
and
$$
f_4(x,z)=21\, \left( {x}^{2}-{\frac {16\,z}{21}}-{\frac {25}{14}} \right) ^{2}+
{\frac {1424}{21} \left( z-{\frac {75}{178}} \right) ^{2}}-{\frac {
21005}{356}}
$$
Thus, the global minimum values of $f_3(x,z)$ and $f_4(x,z)$ are $-7225/84$ and $-21005/356$, respectively. From the above expressions we get that  the real algebraic curves $\{f_3(x,z)=0\}$ and  $\{f_4(x,z)=0\}$ both are formed by two disjoint closed connected components (symmetric respect to the $z$-axis). Therefore, we look for rectangles $[x_3^{-},x_3^{+}]\times[z_3^{-},z_3^{+}]$ and $[x_4^{-},x_4^{+}]\times[z_4^{-},z_4^{+}]$ that contain the algebraic curves $\{f_3(x,z)=0\}$ and  $\{f_4(x,z)=0\}$, respectively.

The smallest and largest values of $x$ and $z$ in the semiplane $\{x\geq 0\}$ for the algebraic curves $\{f_3(x,z)=0\}$  are computed by using the real zeros of the resultants
$$
Res(f_3,(f_3)_z,z)=640000(21x^2-20)(x^2-5)
$$
and
$$
Res(f_3,(f_3)_x,x)=441000000(4z^2+1)(1344z^2-289)^2,
$$
respectively. Since $\sqrt{20/21}<\sqrt{5}<5/2$ and $\sqrt{289/1344}<1/2$ we obtain that $\{f_3(x,z)=0\}$ is contained in the closed rectangle $[-5/2,5/2]\times [-1/2,1/2]$. Moreover,  straightforward computations give that $f_3(x,z)$ restricted to such rectangle reaches the maximum value $3825/16$.  Analogously, we can show that $\{f_4(x,z)=0\}$ is contained in the closed rectangle $[-2,2]\times [-3/5,3/2]$ and $f_4(x,z)$ restricted to this rectangle reaches the maximum value $200$. 

From the above, $f_3(x,z)f_4(x,z)>0$  in the exterior of the closed rectangle  $\mathcal{R}:=[-5/2,5/2]\times [-3/5, 3/2]$. Hence, $c_0>0$ in $\R^2 \backslash \mathcal{R}$.
It remains to prove that $c_0>0$ in  $\mathcal{R}$. 

The polynomials $f_3(x,z)$ and $f_4(x,z)$ restricted to the rectangle $\mathcal{R}$ reach the
maximum values $16625/16$ and $41629/80$, respectively (attained at
$(\pm 5/2, 3/2)$ and $(\pm 5/2, -3/5)$). Therefore, $f_3(x,z)f_4(x,z)$ has a
global minimum value, which is greater than
$$
\min\left\{\left(-\frac{7225}{84}\right)\frac{41629}{80},\;
\left(-\frac{21005}{356}\right)\frac{16625}{16}\right\}
=-\frac{349208125}{5696}.
$$
Thus, for
$$
0<\mu<\mu_0^{*}:=\sqrt{\frac{2\cdot 441\cdot 5696}{3\cdot 5\cdot 349208125}}
=\sqrt{\frac{239232}{249434375}}\approx 0.031,
$$
we have
$$
\frac{5\mu^{2}}{441}\,f_3(x,z)\,f_4(x,z)>-\frac{2}{3}.
$$
Since $1/500<\mu_0^{*}<\mu_2^{*}$, this inequality also holds for
every $0<\mu<1/500$.

We now analyse the function $g(x,z)=ex^{2}+10ez+10$. Since $e>0$, $g$ is even in
$x$, non-decreasing in $|x|$ and increasing in $z$. Hence, on the closed rectangle
$\mathcal{R}=[-5/2,5/2]\times[-3/5,3/2]$ it attains its minimum at $(0,-3/5)$, that is,
\[
g(x,z)\ \ge\ g(0,-3/5)=10-6e \qquad\text{for all } (x,z)\in \mathcal{R} .
\]
In particular, if $0<e\le 1/6$, then $g(x,z)\ge 9$ on $\mathcal{R}$, and therefore
\[
\frac{1}{100}\bigl(g(x,z)\bigr)^{2}\ \ge\ \frac{81}{100}\ >\ \frac{2}{3}
\qquad\text{for all } (x,z)\in \mathcal{R} .
\]
Combining this with the previous estimate we obtain $c_0 \geq (81/100)-(2/3)>0$
on $\mathcal{R}$. Therefore,
$c_0>0$ in $\mathbb{R}^{2}$ for all $0<\mu<\mu_0^{*}$ and
$0<e\le 1/6$.

\medskip
{\bf Step 3.}
We regard $d=c_2c_1-c_0$ as a polynomial in $(x,z,\mu,e)$ and  $(\mu,e)$ varying over the compact rectangle $\mathcal{R}_2:=[1/2000,3/2000]\times[1/20,3/20]$, which contains
$(\mu_0,e_0)=(1/1000, 1/10)$ in its interior.
 Assign weight $1$ to $x$ and weight $2$ to $z$,
so that each monomial $x^{i}z^{j}$ in $d$ 
has weighted degree $i+2j$, and set $s:=x^{4}+z^{2}$. Thus, the weighted degrees of $d_0$,
$d_1$, $d_2$, and $d_3$ are $4$, $8$, $8$, and $12$, respectively. Hence, the highest weighted degree part of $d$ is
the part of $\mu^{3}d_3$ of weighted degree $12$, namely
$$
D_{12}:=\frac{10\mu^{3}}{9261}\left(63x^{4}+400z^{2}\right)
\left(21x^{4}+80z^{2}\right)
\left(21x^{4}+400z^{2}\right).
$$
Then, we can write
$$
d=D_{12}+R,
$$
where $R$ consists of the remaining monomials in $d$, each of which has degree at most four in $z$ and weighted degree at most 10.

On the one hand, the last three factors of $D_{12}$ are bounded below by $63s$, $21s$,
and $21s$, respectively. Thus, we have
\begin{equation}
    \label{primera-desigualdad}
D_{12}\geq 30\,\mu^3\, s^3\geq 
\bigg(\frac{30}{(2000)^3}\bigg)s^3.
\end{equation}
On the other hand,
Since $s=x^{4}+z^{2}$, $|x|\le s^{1/4}$ and $|z|\le s^{1/2}$. Thus, every monomial of weighted degree $w$ is
bounded by $s^{w/4}$. Hence, from the definition of $R$ we get that 
\begin{equation}
    \label{segunda-desigualdad}
    |R|\leq \sum_{j=0}^4\sum_{i=0}^{10-2j}\big{|}C_{ij}(\mu,e)\big{|}|x|^i|z|^j \leq C s^{5/2}, 
    \quad \mbox{for $s\geq 1$.}
\end{equation}
where $C$ is the maximum over $\mathcal{R}_2$ of the sum of the absolute values of the
coefficients of $R$, which satisfies 
$$
C\geq \big{|}C_{00}(\mu,e)\big{|} =\frac{400\,\mu}{21}+\frac{
40000\,{\mu}^{2}}{441}+\frac{2000000\,{\mu}^{3}}{9261}\geq \frac{353641}{37044000}.
$$

Combining the inequalities \eqref{primera-desigualdad} and \eqref{segunda-desigualdad}, 
for each $(\mu,e)\in \mathcal{R}_2$, we have 
$$
d\ \geq \ s^{5/2}\bigg(\frac{30}{(2000)^{3}}\,s^{1/2}-C\bigg)>0
\qquad\text{whenever }
s>s_0:=\bigg(\frac{(2000)^3 C}{30}\bigg)^{2}
 \geq 1.
$$

On the compact set $K:=\{s\le s_0\}$, the proof of
Lemma~3 gives $\delta:=\min_{K}d(\,\cdot\,;\mu_0,e_0)>0$, and by uniform
continuity of $d$ on $K\times P$ there is an open neighbourhood
$U\subset\operatorname{int} \mathcal{R}_2$ of $(\mu_0,e_0)$ with $d>\delta/2$ on $K\times U$.
Therefore, we have proved that $d>0$ on $\mathbb{R}^{2}\times U$.

In summary, from Lemma~\ref{invarianza-de-matriz-jacobiana}, Steps 1, 2 and 3, the Routh--Hurwitz criterion shows that $DF_{\mu,e}(x,y,z)$ is
Hurwitz on $\mathbb{R}^{3}$ for every $(\mu,e)\in U$. 
 \end{proof}

\begin{remark}
\label{rem-nD-vec-field}
For every $n>3$ and $(\mu,e)\in U$, the polynomial vector field
$$
\widetilde{\mathbf{F}}_{\mu,e}(\mathbf{x},\mathbf{y})
:=\bigl(\mathbf{F}_{\mu,e}(\mathbf{x}),-\mathbf{y}\bigr),
\qquad \mathbf{x}\in\R^3,\quad \mathbf{y}\in\R^{n-3},
$$
is Hurwitz on $\R^n$, has the origin as its unique singular point,
which is locally asymptotically stable, and has exactly two non-constant hyperbolic 
$2\pi$-periodic orbits, $\widetilde\gamma_i=\gamma_i\times\{0\}$, $i=1,2$.
$\widetilde\gamma_1$ is of saddle type and $\widetilde\gamma_2$ is locally asymptotically stable. Indeed, the Jacobian matrix is block diagonal:
$$
D\widetilde{\mathbf{F}}_{\mu,e}(\mathbf{x},\mathbf{y})
=\begin{pmatrix}
D\mathbf{F}_{\mu,e}(\mathbf{x})&0\\
0&-I_{n-3}
\end{pmatrix}.
$$
Hence, its eigenvalues are those of $D\mathbf{F}_{\mu,e}(\mathbf{x})$ together with
$-1$, repeated $n-3$ times, so it is Hurwitz everywhere. The equation
$\widetilde{\mathbf{F}}_{\mu,e}(\mathbf{x},\mathbf{y})=0$ forces $\mathbf{y}=0$ and $\mathbf{x}=0$;
the Hurwitz property at the origin implies local asymptotic stability.
Moreover, $\mathbf{y}(t)=\mathbf{y}(0)e^{-t}$, so any periodic solution must satisfy
$\mathbf{y}(t)\equiv 0$. Thus any periodic orbit of $\widetilde{\mathbf{F}}_{\mu,e}(\mathbf{x},\mathbf{y})$  is a periodic orbit of 
$\mathbf{F}_{\mu,e}(\mathbf{x})$. 
Together with Theorem~\ref{main-theorem}, provides  Hurwitz polynomial vector fields in every
dimension $n\ge3,$ having periodic orbits.
\end{remark}

\section{Proofs of Lemmas}
\label{proofs}

\begin{proof}[Proof of Lemma~\ref{propiedades-dinamicas}]
In  cylindrical coordinates the vector field becomes
\[
\mathbf{F}_{\mu,e}(r,\theta,z)=\bigg(-\frac{\mu r}{21}\Big(21r^4-125r^2+400z^2+100\Big),1+e\Big(z+\frac{1}{10}r^2\Big), -z-\frac{1}{10}r^2\bigg).
\]
For $r=0$ (the $z$-axis), $\mathbf{F}_{\mu,e}(r,\theta,z)$ has exactly one equilibrium at the origin. For $r> 0$, a singular point of $\mathbf{F}_{\mu,e}(r,\theta,z)$ requires 
$$
21r^4-125r^2+400z^2+100=0
 \quad \mbox{and} \quad
  -z-\frac{1}{10}r^2=0. 
$$
This algebraic system has only two solutions for $r \geq 0 $:
\begin{equation}
\label{solutions}  
(r_1,z_1)=\Big(1,-\frac{1}{10}\Big)
 \quad \mbox{and} \quad
(r_2,z_2)=\Big(2,-\frac{2}{5}\Big).
\end{equation}
For these two solutions the second component of $\mathbf{F}_{\mu,e}(r,\theta,z)$ becomes $1$. Hence, these two solutions are $2\pi$-periodic orbits of the vector field and they correspond to $\gamma_1$ and $\gamma_2$ in the statement of the lemma. Thus, the origin is the unique singular point of 
$\mathbf{F}_{\mu,e}(x,y,z)$. 

The eigenvalues of the Cartesian Jacobian matrix at the origin are  $-1$ and $\ -(100/21)\mu\pm \sqrt{-1}$. Hence, the origin is locally asymptotically stable.

Since the second component of $\mathbf{F}_{\mu,e}(r,\theta,z)$ becomes $1$ at the periodic orbits $\gamma_1$ and $\gamma_2$, 
the stability of these two orbits only depends of the eigenvalues of the cylindrical Jacobian matrix of the first and third components of $\mathbf{F}_{\mu,e}(r,\theta,z)$, which is
$$
\begin{pmatrix}
-\frac{5 \mu }{21}\left(21r^4-75r^2+20+80z^2\right) & -\frac{800 \mu}{21} rz\\[0.4pc]
-\frac{r}{5} & -1
\end{pmatrix}
$$
The eigenvalues of this Jacobian matrix at $(r_1,z_1)=(1,-1/10)$ are 
$$
 \frac{83 \mu}{21}-\frac{1}{2}+\sqrt{\bigg(\frac{83\mu}{21}\bigg)^2+\frac{67\mu}{21}+\frac{1}{4}}>0
 \quad
 \mbox{and}
 \quad
 \frac{83\mu}{21}-\frac{1}{2}-\sqrt{\bigg(\frac{83\mu}{21}\bigg)^2+\frac{67\mu}{21}+\frac{1}{4}}<0.
$$
Hence, the periodic orbit $\gamma_1$  is  hyperbolic of saddle type.

The eigenvalues of the Jacobian matrix at $(r_2,z_2)=(2,-2/5)$ are   
$$
-\frac{172 \mu}{21}-\frac{1}{2}
+\sqrt{\bigg(\frac{172 \mu}{21}\bigg)^2-\frac{428 \mu}{21}+\frac{1}{4}}
\quad
 \mbox{and}
 \quad
 -\frac{172 \mu}{21}-\frac{1}{2}
-\sqrt{\bigg(\frac{172 \mu}{21}\bigg)^2-\frac{428\mu}{21}+\frac{1}{4}},
$$
whose real part is negative for $0<\mu$.
Hence, the periodic orbit $\gamma_2$  is locally asymptotically stable.

It remains to prove that $F_{\mu,e}(x,y,z)$ has no periodic orbits other than
$\gamma_1$ and $\gamma_2$. The restriction of $F_{\mu,e}(x,y,z)$ to the $z$-axis gives $F_{\mu,e}(0,0,z)=(0,0,-z)$, which implies that the $z$-axis is an invariant straight line of $F_{\mu,e}(x,y,z)$. Hence, any periodic orbit of $F_{\mu,e}(x,y,z)$ is contained in the complement of the $z$-axis. To look for such periodic orbits, we will use $F_{\mu,e}(r,\theta,z)$, with $r>0$.

Since the first and third components of
$F_{\mu,e}(r,\theta,z)$ do not depend on $\theta$, they define the planar
polynomial vector field
\begin{equation}
\label{eq:reduced}
X_{\mu}(r,z):=\left(-\frac{\mu r}{21}\left(21r^{4}-125r^{2}+400z^{2}+100\right),\;
-z-\frac{1}{10}r^{2}\right),
\end{equation}
which we call the reduced vector field. Note that $X_{\mu}(r,z)$ does not depend on
$e$. 
 
Let $\Gamma$ be a $T$-periodic orbit of $F_{\mu,e}(x,y,z)$. The projection $(x,y,z)\to(r,z)$ sends $\Gamma$ onto either a singular point, or a periodic orbit or a polycycle (or graphic) of the planar vector field $X_{\mu}(r,z)$, contained in the right half-plane $\mathbb{H}:=\{r>0\}$, which is invariant under the flow of $X_{\mu}(r,z)$. The singular points of
$X_{\mu}(r,z)$ in $\mathbb{H}$ are exactly $(r_1,z_1)$ and $(r_2,z_2)$ in \eqref{solutions}, which correspond to the periodic orbits $\gamma_1$ and $\gamma_2$, respectively. This fact and the uniqueness of trajectories of $C^1$ vector fields imply that the projection of $\Gamma$ is not be a polycycle. So, it must be a periodic orbit of $X_{\mu}(r,z)$. We
claim that this case cannot occur. Indeed, a direct computation gives
\begin{equation*}
\label{eq:div}
\operatorname{div}X_{\mu}(r,z)
=-1-\mu\Bigg(5\bigg(r^{2}-\frac{25}{14}\bigg)^{2}+\frac{400}{21}z^{2}
-\frac{6575}{588}\Bigg),
\end{equation*}
whence for $0<\mu\leq 588/6575$, we have
$$
\operatorname{div}X_{\mu}(r,z)\le-1+\frac{6575}{588}\,\mu<0
\qquad\text{for all }(r,z)\in\mathbb{H}.
$$
Since $\mathbb{H}$ is simply connected, the Bendixson criterion implies that $X_{\mu}(r,z)$ does not have (non-constant) periodic orbits in $\mathbb{H}$. This proves the
claim. Therefore, $\gamma_1$ and $\gamma_2$ are the only two periodic orbits of
$F_{\mu,e}(x,y,z)$, for 
$0< \mu \leq 588/6575$.
\end{proof}

\begin{proof}[Proof of Lemma~\ref{invarianza-de-matriz-jacobiana}]
Consider the rotation matrix  $R_\theta$
that makes a rotation by an angle $\theta\in [0,2\pi]$ around the $z$-axis. Since $x^2+y^2$ is invariant under these rotations, the field satisfies
\[
\mathbf{F}(R_\theta \mathbf{q})=R_\theta \mathbf{F}(\mathbf{q}), \qquad \mathbf{q}\in\R^3.
\]
Differentiating with respect to $\mathbf{q}$, with $\theta$ fixed, yields
\[
D\mathbf{F}(R_\theta \mathbf{q})R_\theta=R_\theta D\mathbf{F}(\mathbf{q}).
\]
Multiplying by the right  this last equation with $R_\theta^{-1}$, we get
\[
D\mathbf{F}(R_\theta \mathbf{q})=R_\theta D\mathbf{F}(\mathbf{q})R_\theta^{-1}.
\]
Therefore, the characteristic polynomials of $D\mathbf{F}(\mathbf{q})$ and $D\mathbf{F}(R_\theta \mathbf{q})$ are the same. 
\end{proof}

\begin{proof}[Proof of Lemma~\ref{Hurwitz-en-plano-xz}]  For the values $\mu=\mu_0=1/1000$ and $e=e_0=1/10$, the vector field \eqref{campo-principal} becomes

\begin{equation}
\label{campo-principal-part}
\mathbf{F}_{\mu_0,e_0}(x,y,z):=\Big(F_1(x,y,z), F_2(x,y,z), F_3(x,y,z)\Big),
\end{equation}
where
\begin{align*}
F_1(x,y,z)&={\frac {x}{1000} \left(  \left( {x}^{2}+{y}^{2}-5 \right)  \left( {
\frac {20}{21}}-{x}^{2}-{y}^{2} \right) -{\frac {400\,{z}^{2}}{21}}
 \right) }- \left( 1+\frac{z}{10}+{\frac {{x}^{2}}{100}}+{\frac {{y}^{2}}{100}
} \right) y
\\
F_2(x,y,z)&={\frac {y}{1000} \left(  \left( {x}^{2}+{y}^{2}-5 \right)  \left( {
\frac {20}{21}}-{x}^{2}-{y}^{2} \right) -{\frac {400\,{z}^{2}}{21}}
 \right) }+ \left( 1+\frac{z}{10}+{\frac {{x}^{2}}{100}}+{\frac {{y}^{2}}{100}
} \right) x
\\
F_3(x,y,z)&=-z-\frac{1}{10}\,{x}^{2}-\frac{1}{10}\,{y}^{2}
\end{align*}
The characteristic polynomial of $D\mathbf{F}_{\mu_0,e_0}(x,0,z)$ has the form
$$
P_{\mu_0,e_0}(\lambda)=\lambda^3+c_{20}\lambda^2
+c_{10}\lambda+c_{00},
$$
where
\begin{align*}
c_{20}&={\frac {3}{500} \left(x^2-{\frac {125}{63}} \right) ^{2}}+{\frac {4\,{z}^{2}}{105}}+{\frac {26087
}{26460}}
, \\[1pc]
c_{10}&={\frac {44521}{44100}}+\frac{z}{5}+{\frac {2129\,{z}^{2}}{44100}}+{\frac {4\,{
z}^{4}}{11025}}-{\frac {19\,z{x}^{2}}{5250}}+{\frac {709\,{x}^{2}}{
44100}}\\[0.2pc] &\qquad \quad -{\frac {{z}^{2}{x}^{2}}{2205}}+{\frac {37837\,{x}^{4}}{5880000
}}+{\frac {{z}^{2}{x}^{4}}{8750}}-{\frac {{x}^{6}}{21000}}+{\frac {{x}
^{8}}{200000}}
,\\[1pc]
c_{00}&={\frac {44101}{44100}}+\frac{z}{5}+{\frac {449\,{z}^{2}}{44100}}+{\frac {4\,{z
}^{4}}{11025}}-{\frac {8\,{z}^{3}{x}^{2}}{55125}}+{\frac {433\,z{x}^{2
}}{220500}}+{\frac {877\,{x}^{2}}{44100}}\\[0.2pc] &\qquad \quad -{\frac {{z}^{2}{x}^{2}}{2205
}}+{\frac {z{x}^{4}}{22050}}+{\frac {1381\,{x}^{4}}{5880000}}+{\frac {
{z}^{2}{x}^{4}}{8750}}-{\frac {z{x}^{6}}{131250}}-{\frac {{x}^{6}}{
21000}}+{\frac {{x}^{8}}{200000}},
\end{align*}
and $d_{00}=c_{20}c_{10}-c_{00}$, up to a positive constant factor, takes the form
\begin{align*}
d_{00}&=
 12800000\,{z}^{6}+\left( 6048000\,{x}^{4}-24000000\,{x}^{2}+
1706400000 \right) {z}^{4} 
\\[0.2pc] &\qquad 
+\left( 6720000\,{x}^{2}+7056000000
 \right) {z}^{3} + 
 \Big(71322600000-501300000\,{x}^{2}
 \\[0.2pc] &\qquad \qquad
 +506284000\,{
x}^{4}-6720000\,{x}^{6}
+811440\,{x}^{8} \Big) {z}^{2}
+\Big(1764000000
 \\[0.2pc] &\qquad \qquad \quad
 -9612120000\,{x}^{2}+1149120000\,{x}^{4} -
13053600\,{x}^{6}
 \Big) z
 \\[0.2pc] &\qquad \qquad  \qquad
+ 17724200000-25646700000\,{x}^{2}+11053721500\,{x}^{4}
\\[0.2pc] &\qquad  \qquad  \qquad \quad
-52974750
\,{x}^{6}+36850065\,{x}^{8}
-374850\,{x}^{10}+
 27783\,{x}^{12}.
\end{align*}

According to Steps 1 and 2 of the proof of Theorem~\ref{main-theorem}, we need to prove only that $d_{00}>0$ for all $(x,z)\in \R^2$.

By considering $d_{00}$ 
as a polynomial in $z$ with coefficients in $\R[x]$, we compute  the discriminant sequence by using \texttt{DiscriminantSequence} of \textit{Maple}, which has the form
\[
\mathcal{DS}(d_{00})=
\big[
1,-D_2,
-D_3,
D_{4},
D_{5},
-D_{6}
\big],
\]
where
\begin{align*}
D_2:=&\, 63\,x^4-250\,x^2+17775,\\
D_{3}:=& \,111132\,x^{12}-1102500\,x^{10}
+70387800\,x^8-1616076250\,x^6\\&\;
+18925970625\,x^4-362731687500\,x^2
+8294114390625,\\
D_{4}:=&\, 26072900784\,x^{24}
-539101332000\,x^{22}
+54216440775000\,x^{20}\\&\;
-1279135339650000\,x^{18}
+49060940936175000\,x^{16}\\&\; \;
-901852112274750000\,x^{14}
+26740634417559718750\,x^{12}\\&\quad
-29510380031013281250\,x^{10}
+21823983403606930078125\,x^8\\&\quad \;
+126307747804471611328125\,x^6
+10174454329692519931640625\,x^4\\&\quad \; \;
+31486855429697047119140625\,x^2
+1191580881794092278808593750,\\
\end{align*}
$D_5$ and $D_6$ are even polynomials with integer coefficients of degrees 40 and 60, respectively. Moreover, the leading coefficient of every $D_i$, $i=2\ldots,6$, is positive.
By applying again the \texttt{DiscriminatSequence} to each polynomial $D_2,\ldots,D_6$, respect to $x$, we can determine the sign list of the discriminant sequence  $\mathcal{DS}(d_{00})$. For instance, the discriminant sequences of the polynomial $D_2$ of degree $4$, the polynomial $D_3$ of degree $12$,  and the polynomial $D_4$ of degree 24 are
$$
[1, 1, -1, 1],
\qquad
[1, 1, -1, 1, 1, 1, -1, 1, -1, 1, 1, 1]
$$
and
$$
[1, 1, -1, 1, 1, -1, 1, 1, -1, -1, -1, 1, -1, -1, -1, 1, 1, -1, 1, 1, 1, -1, 1, 1],
$$
respectively. The first sequence
has $2$ changes of sign, the second one has $6$ changes of sign, and the third one has $12$ changes of sign, that implies, according to \cite[Theorem 2.1]{Yang}, that $D_2$, $D_3$ and $D_4$ do not have any real zero. Hence, $D_2$, $D_3$ and $D_4$ are positive on $\R$.

Analogously, we compute the discriminant sequences of $D_5$ and $D_6$, obtaining that $D_5$ has only two real zeros and that $D_6$ is positive.
Thus, for each fixed $x\in \R$ , we obtain that the discriminant sequence  $\mathcal{DS}(d_{00})$ has one of the three forms:
$$
[1,-1,-1,1,1,-1] ,
\quad
[1,-1,-1,1,0,-1] ,
\quad
\mbox{or}
\quad
[1,-1,-1,1,-1,-1] ,
$$
The revised sign list (see \cite{Yang}) in the second case becomes
$$
[1,-1,-1,1,-1,-1] .
$$

Consequently, in all three cases, the revised sign list
has three sign changes. By \cite[Theorem~2.1]{Yang},
for every fixed $x\in\R$, the polynomial
$z\mapsto d_{00}(x,z)$ has no real zeros.
Since its leading coefficient is positive, we conclude
that $d_{00}(x,z)>0$ for every $(x,z)\in\R^2$.
\end{proof}

\newpage

\section{Numerical Illustration}\label{example}

In this section, we illustrate the dynamical behavior of the polynomial vector field  $F_{\mu_0,e_0}(x,y,z)$ given in \eqref{campo-principal-part} . The phase portraits provide a geometric interpretation of the coexistence of a unique asymptotically stable equilibrium and two hyperbolic $2 \pi-$periodic orbits, one of saddle type and the other attracting, despite the global Hurwitz property of the Jacobian matrix.

\begin{figure}[!htbp]
    \centering
\includegraphics[width=\textwidth,height=0.50\textheight,keepaspectratio]
    {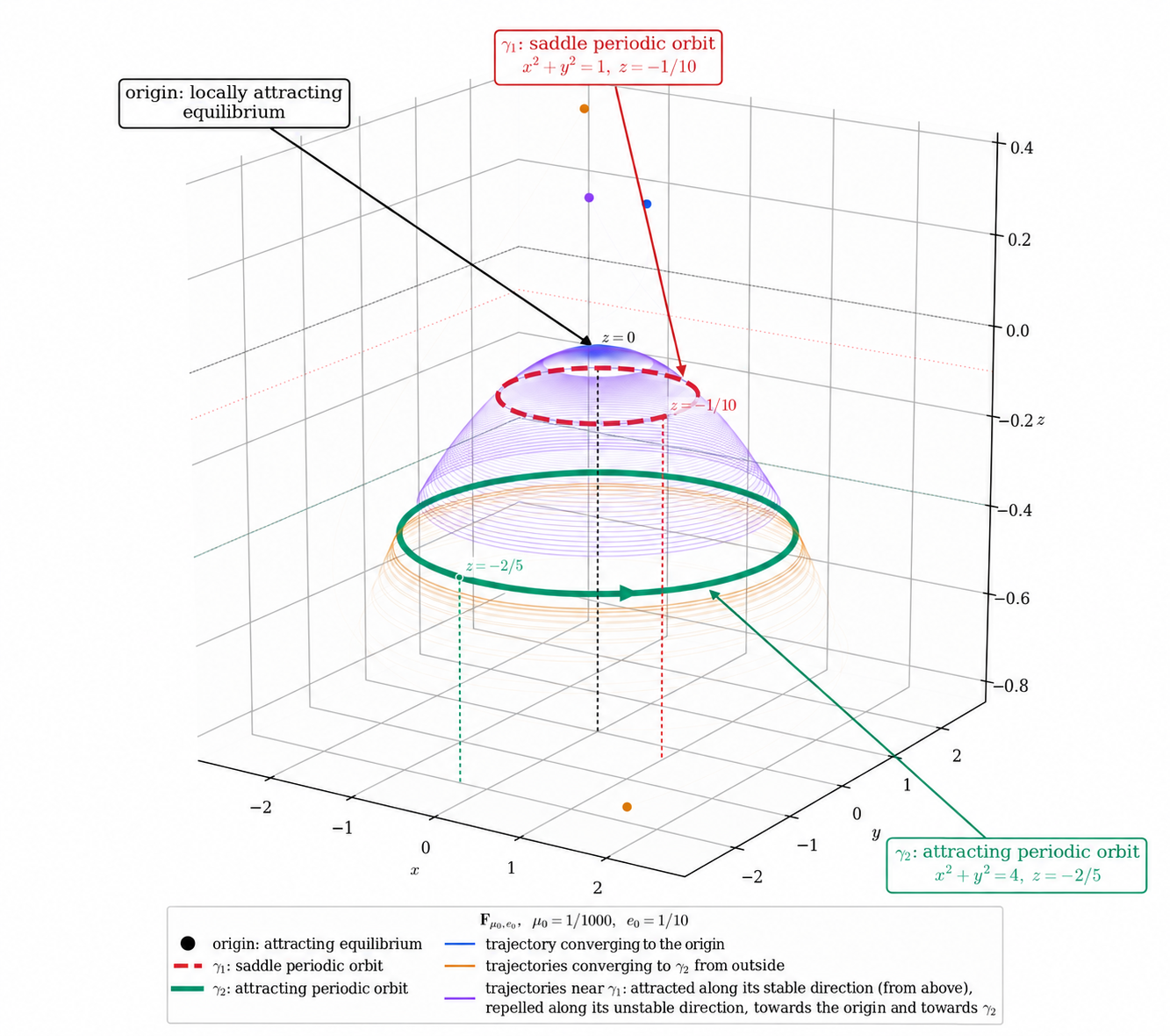}
    \caption{Phase portrait of the polynomial vector field $F_{\mu_0,e_0}(x,y,z)$.
    }
    \label{fig:phase-portrait-polynomial-field}
\end{figure}

\begin{figure}[!htbp]
    \centering
\includegraphics[width=\textwidth,height=0.6\textheight,keepaspectratio]
    {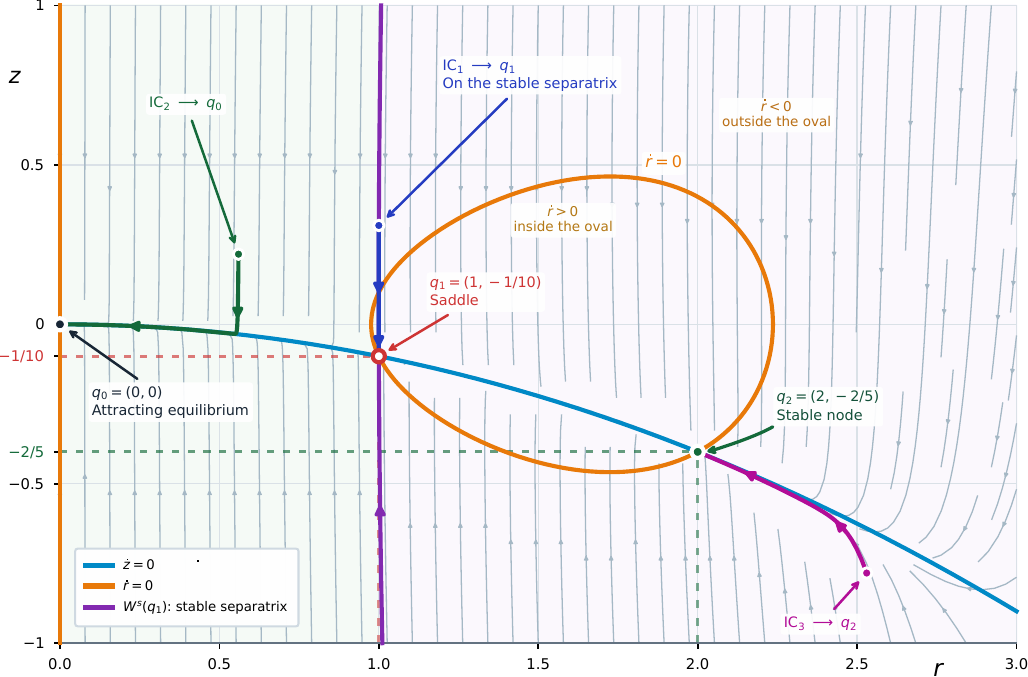}
    \caption{Phase portrait of the reduced vector field $X_{\mu}(r,z)$, given in \eqref{eq:reduced}, for $\mu=\mu_0=1/1000$ (it does not depend
on $e=e_0=1/10$). The origin
is an attractor and corresponds to the singular point of $F_{\mu_0,e_0}$. The  singular point $q_1=(1,-1/10)$ is a saddle and corresponds to the
saddle periodic orbit $\gamma_1$; the singular point $q_2=(2,-2/5)$ is a stable
node and corresponds to the attracting periodic orbit $\gamma_2$; The blue
curve is the null-cline $\{\dot z=0\}$, which is the parabola $\{z=-r^{2}/10\}$, and the orange curve is
the null-cline $\{\dot r=0\}$, which is the oval $\{21r^{4}-125r^{2}+400z^{2}+100=0\}$; inside the oval
$r$ increases and outside it decreases. The purple curve is the stable
separatrix, $W^{s}(q_1)$, of the saddle.
Although it looks like the vertical line $\{r=1\}$, that line is not
invariant; the separatrix is
the graph of $r=1+\frac{200\mu}{21}\bigl[(z-\frac{1}{10})^{2}-\frac{1}{25}\bigr]
+O(\mu^{2})$, which for $|z|\le1$ differs from $r=1$ by less than $0.011$,
hence its apparent straightness. The separatrix  $W^{s}(q_1)$ splits the right half-plane $\{r>0\}$
into the basin of attraction of the origin (to its left) and the basin of
attraction of $q_2$ (to its right); rotating it around the $z$-axis one
obtains the stable manifold $W^{s}(\gamma_1)$, which separates the basin of
the origin from the basin of $\gamma_2$ in $\mathbb{R}^{3}$. The solutions with initial conditions
$\mathrm{IC}_1$, $\mathrm{IC}_2$ and $\mathrm{IC}_3$ converge, respectively,
to the saddle along its stable separatrix, to the origin, and to the
attracting periodic orbit.}
\end{figure}

\FloatBarrier

\section*{Statement on the use of generative artificial intelligence}

Generative artificial intelligence played a role in the exploratory stage of this work.
The construction presented in this paper emerged from an iterative and exploratory dialogue 
between the authors and OpenAI's ChatGPT. This dialogue was guided by mathematical questions, constraints, strategies, reformulations, and
additional mathematical requirements, all based on the author's mathematical experience and previous results.
In this process, ChatGPT served as an exploratory tool for exploring, suggesting, and testing candidate vector fields that might satisfy the desired properties, namely Hurwitz and the existence of periodic orbits.

AI-generated suggestions, symbolic manipulations, and numerical observations were not treated as mathematical evidence. The final construction, mathematical validation, and all proofs were carried out by the authors. In particular, the global Hurwitz property was proved analytically using exact inequalities and the Routh--Hurwitz criterion.  All mathematical results in the paper were proved and organized by the authors, who take full responsibility for every theorem, computation, and conclusion presented in the manuscript.

\end{document}